\documentclass[letterpaper, 10 pt, conference]{ieeeconf}  

\usepackage{graphicx,algorithm,algorithmic,amsthm,amsmath,amsfonts,verbatim,mathtools,enumerate,bm,hyperref,tikz,textcomp}
\usetikzlibrary{arrows}
\usepackage[caption=false]{subfig}

\usepackage{nomencl}
\makenomenclature
\usepackage{etoolbox}
\renewcommand\nomgroup[1]{%
  \item[\bfseries
  \ifstrequal{#1}{A}{Parameters}{%
  \ifstrequal{#1}{B}{Variables}{%
  \ifstrequal{#1}{C}{Subscripts}{%
  \ifstrequal{#1}{D}{Sets}{%
  \ifstrequal{#1}{E}{Operations}{%
  }}}}}%
]}

\allowdisplaybreaks

\newcommand{\tr}{\mathop{\rm tr}}

\newcommand{\vect}{\mathop{\rm vec}}

\newcommand{\mnorm}[1]{{\left\vert\kern-0.25ex\left\vert\kern-0.25ex\left\vert #1 
    \right\vert\kern-0.25ex\right\vert\kern-0.25ex\right\vert}}

\newtheorem{definition}{Definition}

\newtheorem{corollary}{Corollary}
\newtheorem{remark}{Remark}
\newtheorem{proposition}{Proposition}
\newtheorem{assumption}{Assumption}
\newtheorem{example}{Example}

\newcommand{\ie}{{\it i.e.}}

\def\BibTeX{{\rm B\kern-.05em{\sc i\kern-.025em b}\kern-.08em
    T\kern-.1667em\lower.7ex\hbox{E}\kern-.125emX}}
\begin{document}

\title{Vertiport Selection in Hybrid Air-Ground Transportation Networks via Mathematical Programs with Equilibrium Constraints
}

\author{Yue~Yu, Mengyuan~Wang, Mehran~Mesbahi, and Ufuk~Topcu
\thanks{Y. Yu and U. Topcu are with the Oden Institude of Computational Sciences and Engineering, The University of Texas at Austin, TX 78712 USA (emails: yueyu@utexas.edu, \,utopcu@utexas.edu). M. Wang and M. Mesbahi are with the Department of Aeronautics and Astronautics, University of Washington, Seattle, WA 98195 USA (emails: mywang10@uw.edu, \,mesbahi@uw.edu).}
}
\maketitle

\begin{abstract}
Urban air mobility is a concept that promotes aerial modes of transport in urban areas. In these areas, the location and capacity of the vertiports--where the travelers embark and disembark the aircraft--not only affect the flight delays of the aircraft, but can also aggravate the congestion of ground vehicles by creating extra ground travel demands. We introduce a mathematical model for selecting the location and capacity of the vertiports that minimizes the traffic congestion in hybrid air-ground transportation networks. Our model is based on a mathematical program with bilinear equilibrium constraints. Furthermore, we show how to compute a global optimal solution of this mathematical program by solving a mixed integer linear program. We demonstrate our results via the Anaheim transportation network model, which contains more than 400 nodes and 900 links.
\end{abstract}


\nomenclature[A,01]{\(n_\tau\)}{Number of different equilibria}
\nomenclature[A,02]{\(n_n\)}{Number of nodes}
\nomenclature[A,03]{\(n_v\)}{Number of candidate vertiports}
\nomenclature[A,04]{\(n_l\)}{Number of ground and air links}
\nomenclature[A,05]{\(n_d\)}{Number of destination nodes}
\nomenclature[A,06]{\(n_g\)}{Number of ground links}
\nomenclature[A,07]{\(n_a\)}{Number of air links}
\nomenclature[A,08]{\(n_c\)}{Number of options for vertiport capacity}
\nomenclature[A,09]{\(n_b\)}{Number of logical constraints}

\nomenclature[A,10]{\(\gamma\)}{Budget for selecting vertiport capacity}
\nomenclature[A,11]{\(\mu\)}{A large positive scalar}
\nomenclature[A,12]{\(\omega\)}{Weighting parameter for selection cost}
\nomenclature[A,13]{\(\mathbf{1}_b\)}{The \(n_b\)-dimensional vector of all 1's}
\nomenclature[A,14]{\(\mathbf{1}_d\)}{The \(n_d\)-dimensional vector of all 1's}
\nomenclature[A,15]{\(\mathbf{1}_c\)}{The \(n_c\)-dimensional vector of all 1's}
\nomenclature[A,16]{\(\mathbf{1}_v\)}{The \(n_v\)-dimensional vector of all 1's}
\nomenclature[A,17]{\(c\)}{Free travel time vector}
\nomenclature[A,18]{\(f\)}{Link capacity vector}
\nomenclature[A,19]{\(w\)}{Value vector for unit vertiport capacity}
\nomenclature[A,20]{\(E\)}{Node-edge incidence matrix}
\nomenclature[A,21]{\(D\)}{vertiport incidence matrix}
\nomenclature[A,22]{\(S\)}{Source-sink matrix}
\nomenclature[A,23]{\(G\)}{Candidate capacity matrix}
\nomenclature[A,24]{\(K\)}{Cost matrix}

\nomenclature[B,01]{\(g\)}{Vertiport capacity vector variable}
\nomenclature[B,02]{\(p\)}{Dual variable for link capacity constraint}
\nomenclature[B,03]{\(q\)}{Dual variable for vertiport capacity constraint}
\nomenclature[B,04]{\(B\)}{Binary selection matrix}
\nomenclature[B,05]{\(U\)}{Dual variable for nonnegative flow constraint}
\nomenclature[B,06]{\(V\)}{Dual variable for flow conservation constraint}
\nomenclature[B,07]{\(X\)}{Flow matrix variable}
\nomenclature[B,08]{\(Y\)}{Auxiliary variable for linearizing constraints}

\nomenclature[C,01]{\([a]_j\)}{The \(j\)-th element of vector \(a\)}
\nomenclature[C,02]{\([A]_{ij}\)}{The \(ij\)-th element of matrix \(A\)}

\nomenclature[D,01]{\(\mathbb{R}^n\)}{The set of \(n\)-dimensional real vectors}
\nomenclature[D,02]{\(\mathbb{R}^{m\times n}\)}{The set of \(m\times n\) real matrices}
\nomenclature[D,03]{\(\{0, 1\}^n\)}{The set of \(n\)-dimensional binary vectors}
\nomenclature[D,03]{\(\{0, 1\}^{m\times n}\)}{The set of \(m\times n\) binary matrices}

\nomenclature[E,01]{\(\odot\)}{The Hadamard product }
\nomenclature[E,02]{\(\vect\)}{The vectorization function }

\printnomenclature
\section{Introduction}
\label{sec: introduction}

Urban Air Mobility (UAM) is a concept that promotes short-range aerial travel in urban areas \cite{garrow2021urban,sun2021operational}. By adding alternative air modes of transportation--mainly supported by electric vertical-take-off-and-landing (eVTOL) aircraft --to the existing ground transportation networks, UAM has the potential to alleviate ground traffic congestion. The latter has become a growing concern among travelers and transportation authorities alike \cite{venkatesh2020optimal}.

An integral part of the operation of eVTOL aircraft is to build vertiports where passengers or cargo embark and disembark the aircraft. Many UAM industries, including Ferrovial, Urban-Air Port Ltd., and Skyports, are actively investigating the possibility of ultra-compact, rapidly deployable, multi-functional vertiports for both manned and unmanned aircraft around the world \cite{aamreport}.

One challenge in UAM is to select the locations of the vertiports optimally among candidate options. The input of the selection includes a set of candidate vertiports (usually generated by clustering algorithms \cite{lim2019selection,rajendran2019insights,jeong2021selection}) and the budget of the total number of vertiports. The output of the selection is a set of selected vertiports that optimizes certain performance metric of the air transportation network supported by the selected vertiports, such as the savings in total travel time \cite{daskilewicz2018progress} and the package demand served by the aircraft \cite{german2018cargo}. 

Since the candidate locations for vertiports are often limited by safety, accessibility, and noise emission factors, most door-to-door travel in UAM will require transportation via both aircraft and ground vehicles  \cite{lim2019selection,rajendran2019insights,jeong2021selection,antcliff2016silicon}. Consequently, the vertiport locations can also affect the traffic in the existing ground transportation network by creating additional travel demands on the way to and from the vertiports. Recent studies investigate the potential of UAM as a complement to the ground transportation systems \cite{ploetner2020long}. However, how to select the vertiport locations by optimizing its impacts on the congestion in a network that allows both aerial and ground modes of transport is, to our best knowledge, still an open question.

There have been various studies on transportation network design based on mathematical programs with equilibrium constraints (MPEC), a nonconvex optimization problem with bilinear complementarity constraints; see \cite{yang1998models,migdalas1995bilevel,farahani2013review} and references therein. However, the existing results have the following limitations when applied to veriport selection. First, to our best knowledge, these results are all based on traffic equilibrium model that only include link parameters, such as the number of lanes and speed limit in a ground transportation networks \cite{patriksson2015traffic,larsson1999side,nesterov2000stable,nesterov2003stationary,chudak2007static}, but not node parameters, such as the number of touchdown-and-lift-off pads and the total number of scheduled flights at a vertiport in an air transportation network. Consequently, despite the success of link parameter design, such as the addition or expansion of candidate road segments, the design of node parameters, such as the location and capacity of the vertiports, has not been investigated in depth. Second, the number of bilinear complementarity constraints in the MPEC for transportation network design depends on the number of nodes and links in the network, which grows rapidly as the network size increases \cite{migdalas1995bilevel}.

We introduce a mathematical model for selecting the location and capacity of the vertiports in a hybrid air-ground transportation network. Our objective is to reduce the congestion in the network. Our contributions are threefold.
\begin{enumerate}
    \item First, we developed a linear-program-based model for the static traffic equilibria in a hybrid air-ground transportation network. This model extends the Nesterov \& de Palma model by adding node capacity constraints. 
    \item Second, we proposed a mathematical program with bilinear equilibrium constraints that optimizes the location and capacity of the vertiports subject to budget and logical constraints. In addition, we showed how to compute a global optimal solution of this mathematical program by solving a mixed-integer linear program (MILP). This MILP does not contain any bilinear complementarity constraints, and the number of integer variables only depends on the number of candidate vertiports, which is typically much smaller than the total number of nodes or links of the network.
    \item Finally, we demonstrated our results using the Anaheim transportation network,  which contains more than 400 nodes, 900 links, and 9 candidate vertiports where each vertiport has two candidate capacity values. Our MILP contains only 18 binary integer variables.
\end{enumerate}
Our work is the first step to adapt the mathematical tools for ground transportation network analysis and design in the age of UAM. In particular, we showed how to extend the ground traffic equilibria model to predict traffic equilibria in hybrid air-ground transportation networks. We also provide a numerical tool to design the vertiports as an extension of an existing ground transportation network.

The rest of the paper is organized as follows. Section~\ref{sec: related work} briefly reviews some existing results in static traffic equilibria and transportation network design. Section~\ref{sec: nesterov} introduces an extended Nesterov \& de Palma model for static traffic equilbria in hybrid air-ground transportation networks. Section~\ref{sec: mixed integer} introduces the mathematical program with equilibrium constraints for vertiports selection. We demonstrate this mathematical program using the Anaheim transportation network in Section~\ref{sec: experiment} before concluding in Section~\ref{sec: conclusion}.

\section{Related work}
\label{sec: related work}

Transportation network design is the problem of determining the optimal modification of an existing ground transportation network \cite{yang1998models,migdalas1995bilevel}. These modifications can be expanding the capacity of existing links, or adding new links to the network. The quality of the modifications is evaluated via the congestion of the travelers in the modified network and the cost of the modifications. The input of the problem includes 1) the existing transportation network topology, 2) the travel demand between each origin-destination pair for a specific time interval, 3) the characteristics of roads, such as flow capacity and free travel time, 4) the set of candidate options for modifications and their cost, and 5) the total budget for modifications. The outcome of the problem is a set of modifications that satisfies the budget constraint and minimizes the congestion of the travelers. See \cite{farahani2013review} for a recent survey on transportation network design.

A key step in transportation network design is to predict the collective behavior of selfish travelers in congested transportation networks \cite{patriksson2015traffic}. There are two different prediction model used in the literature: a nonlinear-convex-optimization-based model, known as the Beckmann model \cite{beckmann1956studies}, and a linear-program-based model, known as the Nesterov \& de Palma model \cite{nesterov2000stable,nesterov2003stationary}. When combined with the Bureau of Public Roads function for link delays, the Beckmann model provides prediction results similar to the Nesterov \& de Palma model, in terms of the user distribution, the price of anarchy, and the Braess paradox phenomenon. We refer interested readers to \cite{chudak2007static} for a detailed numerical comparison of the two models.

In transportation network design problems, the Nesterov \& de Palma model is computationally more efficient than the Beckmann model. The reason is that a network design problem requires not only the prediction of traffic patterns, but also optimizing the predicted traffic patterns by designing the network parameters. The former only requires solving a convex optimization problem; the latter, however, is a nonconvex optimization problem whose constraints include the Karush–Kuhn–Tucker (KKT) conditions of a convex optimization problem \cite{migdalas1995bilevel}. The KKT conditions of the convex optimization in the Beckmann model contain nonlinear equality constraints \cite{migdalas1995bilevel}, whereas KKT conditions of the optimization in the Nesterov \& de Palma model contain linear constraints only \cite{wei2017expansion}. As a result, using the Nesterov \& de Palma model, the network design optimization is equivalent to a mixed integer linear program \cite{wei2017expansion,wang2018coordinated,xia2021distributed}. In contrast, using the Beckmann's model, solving a mixed integer linear program--or equivalently, a linear-linear bilevel optimization problem--only provides a local descent direction, not a global optimal solution, for the network design optimization \cite{patriksson2002mathematical}.

\section{Static traffic equilibria in hybrid air-ground transportation networks}
\label{sec: nesterov}

We first introduce the static traffic equilibria model of a hybrid air-ground transportation network. This model predicts the static regime of the traffic patterns, where the number of travelers entering and exiting a road segment (or a flight leg) per unit time are the same. We will later use this model to evaluate the performance of a given transportation network. Our model is based on the following three assumptions.

\begin{enumerate}
    \item For each origin-destination pair, only the routes with the minimum accumulated travel time are used.
    \item The traffic flow on each road segment or flight leg never exceeds its capacity; the total incoming and outgoing air traffic flow at each vertiport never exceeds its capacity.
    \item If the ground traffic flow on a road segment is below its capacity, its travel time equals to a nominal value; if the capacity is reached, the travel time is higher than the nominal value. If the total incoming and outgoing air traffic flow at a vertiport is below its capacity, the delay (for embarkation and disembarkation) at this vertiport is zero; if the capacity is reached, the delay is nonnegative.
\end{enumerate}
The above three assumptions have the following implications. The first assumption characterize the selfish and competitive nature of the travelers' behavior; it is also known as the \emph{Wardrop equilibrium principle} \cite{wardrop1952road} and has been the basis of static traffic equilibria models \cite{patriksson2015traffic}. The second assumption states that the traffic flow on a road segment is upper bounded--typically due to the number of lanes and the green light time--and the total air traffic flow at a vertiport is upper bounded--typically due to the number of touchdown and lift-off pads. The third assumption is based on the empirical observation that the travel time on a road segment (or the delay at a vertiport) is at its minimum when there is no congestion, and increase with the congestion level. Similar assumptions were first introduced in the Nesterov \& de Palma model for ground traffic equilibria \cite{nesterov2003stationary}. Here we add two additional assumptions on the capacity and delay at vertiports.

\begin{remark}
Notice that the first assumption above is not reasonable at all when understood literally: other than the time consumed during travel, operating cost, such as the fare of a trip, is also an important factor that affects the travelers' decision. However, one can convert such an operating cost to an additional \emph{effective time} using the travelers' average value of time based on their annual income; a similar conversion was used in \cite{roy2021future}. Therefore, without loss of generality, we refer to the term ``travel time" as an \emph{effective travel time} that accounts for both the operating costs and the actual time of travel.
\end{remark}

In the following, we will introduce a mathematical model for static traffic equilibria that satisfy all the aforementioned assumptions. Our model is based on the Nesterov \& de Palma model for ground traffic equilibria.

\subsection{Hybrid air-ground transportation networks}
We first introduce some basic network concepts: nodes, links, incidence matrices, link and node capacity, and travel time. 

\subsubsection{Nodes and links}
We let \(\mathcal{N}=\{1, 2, \ldots, n_n\}\) denote the set of nodes. We let \(\mathcal{V}=\{v(1), v(2), \ldots, v(n_v)\}\) denote the set of nodes that contain a candidate vertiport location, where \(v(i)\in\mathcal{N}\) for all \(i=1, 2, \ldots, n_v\).

We let \(\mathcal{L}=\{1, 2, \ldots, n_l\}\) denote the set of links. Each link is an ordered pair of distinct nodes, where the first and second nodes are the ``tail" and ``head" of the link, respectively. In addition, we let \(n_g\leq n_l\) denotes the number of ground links, and \(n_a\coloneqq n_l-n_g\) denote the number of air links. The presence of link \(k=(i, j)\) with \(1\leq k\leq n_g\) means that any ground travelers can travel from node \(i\) to node \(j\), and the presence of link \(k=(i, j)\) with \(n_g+1\leq k\leq n_l\) means any aircraft can fly from node \(i\) to node \(j\). 

\subsubsection{Incidence matrices}
We represent the topology of the hybrid air-ground network using the node-edge incidence matrix \(E\in\mathbb{R}^{n_n\times n_l}\). The entry \([E]_{ik}\) in matrix \(E\) is associated with node \(i\) and link \(k\) as follows:
\begin{equation}
     [E]_{ik}=\begin{cases}
    1, & \text{if node \(i\) is the tail of link \(k\),}\\
    -1, & \text{if node \(i\) is the head of link \(k\),}\\
    0, & \text{otherwise.}
    \end{cases}
\end{equation}
Note that \([E]_{ik}\neq 0\) for some \(n_g+1\leq k\leq n_l\) only if \(i\in\mathcal{V}\). 

We represent the topology of the air links and vertiports using the following \emph{unsigned incidence matrix} \(D\in\mathbb{R}^{n_v\times n_l}\) for air links. The entry \([D]_{ik}\) is associated with node \(i\) and link \(k\) as follows 
\begin{equation}
     [D]_{ik}=\begin{cases}
    1, & \text{if \(k\geq n_g+1\) and \([E]_{v(i),k}\neq 0\),}\\
    0, & \text{otherwise.}
    \end{cases}
\end{equation}

\subsubsection{Demand matrix}
We distinguish different travelers in the network using their destinations, denoted by a subset of nodes \(\{s(1), s(2), \ldots, s(n_d)\}\subset\mathcal{N}\). We denote the amount of trips per unit time, also known as the \emph{traffic demand}, between different origin and destination nodes using a \emph{demand matrix} \(S\in\mathbb{R}^{n_n\times n_d}\) defined as follows. For any \(i\in\mathcal{N}\) with \(i\neq s(j)\), we let the entry \([S]_{ij}\) in matrix \(S\) denote the traffic demand from node \(i\) to node \(s(j)\), \ie, the amount of travelers leaving node \(i\) heading towards node \(s(j)\) per unit time. If \([S]_{ij}>0\), then \((i, s(j))\) is also known as an \emph{origin-destination pair}. Finally, we let \([S]_{s(j),j}=-\sum_{i, i\neq s(j)}[S]_{ij}\) for all \(j=1, 2, \ldots, n_d\) such that the sum of each column in matrix \(S\) equals zero. Such an assignment is convenient for defining the flow conservation constraints in matrix form, as we will show. 

\subsubsection{Flow matrix}

At a static traffic equilibrium, the amount of travelers entering and exiting the same link are the same. We represent the amount of travelers on different links per unit time using the \emph{flow matrix} \(X\in\mathbb{R}^{n_l\times n_d}\). In particular, the entry \([X]_{kj}\) in matrix \(X\) denotes the amount of travelers exiting link \(k\) while heading towards destination node \(s(j)\) per unit time. 

By construction, the demand matrix \(S\), flow matrix \(X\), and incidence matrices \(E\) together satisfy the following \emph{flow conservation constraint}:
\begin{equation}\label{eqn: flow conserv}
    EX=S, \enskip X\geq 0.
\end{equation}
Notice that the above constraints implicitly imply that the sum of each column in matrix \(S\) equals zero. This observation justifies our definition of the negative entries in matrix \(S\).

\begin{example}
To illustrate the aforementioned network concepts, we consider the example network in Fig.~\ref{fig: network}. In this case, we have \(\mathcal{N}=\{1, 2, 3, 4\}\), \(\mathcal{V}=\{2, 3\}\), \(\mathcal{L}=\{1, 2, 3, 4, 5, 6\}\), and matrices \(E\) and \(D\) are as follows
\begin{equation*}
    E=\begin{bsmallmatrix}
    1 & 1 & 0 & 0 & 0 & 0\\
    -1 & 0 & 1 & 0 & 1 & -1\\
    0 & -1 & 0 & 1 & -1 & 1\\
    0 & 0 & -1 & -1 & 0 & 0
    \end{bsmallmatrix}, \enskip D=\begin{bsmallmatrix}
    0 & 0 & 0 & 0 & 1 & 1\\
    0 & 0 & 0 & 0 & 1 & 1
    \end{bsmallmatrix}.
\end{equation*}
Furthermore, a possible choice of demand matrix \(S\) and flow matrix \(X\) that satisfy the constraints in \eqref{eqn: flow conserv}, are as follows:
\begin{equation*}
   S=\begin{bsmallmatrix}
     5 & -5 & 0 & 0\\
     10 & 0 & 0 & -10
    \end{bsmallmatrix}^\top, \enskip X=\begin{bsmallmatrix}
    5 & 0 & 0 & 0 & 0 & 0\\
    3 & 7 & 7 & 3 & 0 & 4
    \end{bsmallmatrix}^\top.
\end{equation*}
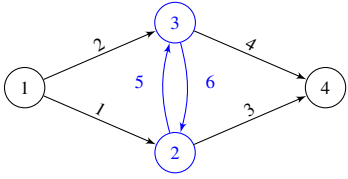
\begin{figure}
    \centering
    \tikzset{
    node/.style={circle,draw,minimum size=0.5em},
    link/.style={->,> = latex'}
}
    \begin{tikzpicture}[scale=0.5]
    
   \node[node] at (0, 0) (1) {\scriptsize 1}; 
   \node[node, blue] at (4, -1.73) (2) {\scriptsize 2}; 
   \node[node, blue] at (4, 1.73) (3) {\scriptsize 3}; 
   \node[node] at (8, 0) (4) {\scriptsize 4}; 
   
   \draw[link] (1) -- node[midway, label={[label distance=-0.5cm,text depth=2ex, rotate=-30]above: {\scriptsize 1}}] {} (2);
   \draw[link] (1) -- node[midway, label={[label distance=-0.5cm,text depth=2ex, rotate=30]above: {\scriptsize 2}}] {} (3);
   \draw[link] (2) -- node[midway, label={[label distance=-0.5cm,text depth=2ex, rotate=30]above: {\scriptsize 3}}] {} (4);
   \draw[link] (3) -- node[midway, label={[label distance=-0.5cm,text depth=2ex, rotate=-30]above: {\scriptsize 4}}] {} (4);
   \draw[link, blue] (2) to[bend left=15]node[midway, label={[label distance=0,text depth=1ex]left: {\scriptsize\color{blue} 5}}] {} (3);
   \draw[link, blue] (3) to[bend left=15]node[midway, label={[label distance=0,text depth=1ex]right: {\scriptsize\color{blue} 6}}] {} (2);

    \end{tikzpicture} 
    \caption{An example of a hybrid air-ground transportation network, where black and blue arcs denote ground and air links, respectively. The blue nodes contain candidate locations for vertiports.}
    \label{fig: network}
\end{figure}
\end{example}

\subsubsection{Capacity and free travel time}
The \emph{link capacity} of a (ground or air) link is the maximum amount of travelers existing this link per unit time. For a ground link, this capacity depends on the number of lanes and cycle time of traffic signals; for an air link, this capacity depends on the available airspace and the maximum allowed aircraft density of each flight leg. We denote the link capacity of all the air and ground links using the \emph{link capacity vector} \(f\in\mathbb{R}_+^{n_l}\), where its entry \([f]_k\) denotes the capacity on link \(k\).

The \emph{free travel time} of a (ground or air) link is the time consumed by each traveler on this link when there is no traffic congestion. We denote the free travel time of all links using a vector \(c\in\mathbb{R}^{n_l}\), whose \(k\)-th entry denotes the free travel time on link \(k\). 

The Nesterov \& de Palma model \cite{nesterov2000stable,nesterov2003stationary} assumes that the link capacity, the free travel time, and the flow matrix are coupled as follows. First, the traffic flow on each link never exceeds its capacity, \ie,
\begin{equation}\label{eqn: link cap}
    \sum_{j=1}^{n_d}[X]_{kj}\leq [f]_k,
\end{equation}
for all \(k\in\mathcal{L}\). Second, if the traffic flow on a link is below its capacity, then the travel time of this link equals the corresponding free travel time; if the traffic flow on a link equals its capacity, then the average travel time of this link is lower bounded by the corresponding free travel time. In other words, if vector \(\tilde{c}\in\mathbb{R}^{n_l}\) is such that \([\tilde{c}]_k\) denotes the travel time on link \(k\), then the following conditions hold for all \(k\in\mathcal{L}\): 
\begin{subequations}\label{eqn: link cs}
\begin{align}
     &\sum_{j=1}^{n_d}[X]_{kj}<[f]_k \Rightarrow [\tilde{c}]_k=[c]_k,\\
     &\sum_{j=1}^{n_d}[X]_{kj}=[f]_k \Rightarrow [\tilde{c}]_k\geq [c]_k.
\end{align}
\end{subequations}

In addition to the above link capacity, here we also consider additional capacity of vertiports in the hybrid air-ground transportation network. Each vertiport can accommodate a maximum amount of take-off and landing per unit time, due to the limited number of touch-down and lift-off pads. Similar to those in Nesterov \& de palma model, we make the following assumptions. First, the total amount of air traffic entering and exiting a vertiport never exceeds its capacity, \ie, 
\begin{equation}\label{eqn: node cap}
    \sum_{k=1}^{n_l}\sum_{j=1}^{n_d}[D]_{ik}[X]_{kj}\leq [g]_i.
\end{equation}
for all \(i=1, 2, \ldots, n_v\).
Second, if the traffic flow on a vertiport is below its capacity, then the delay at this vertiport equals zero; if the traffic flow on a vertiport reaches its capacity, then the average flight delay at this vertiport is nonnegative. In other words, if vector \(\tilde{e}\in\mathbb{R}^{n_v}\) is such that \([\tilde{e}]_i\) denote the average flight delay at vertiport \(i\), then the following condition holds for all \(i=1, 2, \ldots, n_v\): 
\begin{subequations}\label{eqn: node cs}
\begin{align}
     &\sum_{k=1}^{n_l}\sum_{j=1}^{n_d}[D]_{ik}[X]_{kj}<[g]_i \Rightarrow [\tilde{e}]_i=0,\\
     &\sum_{k=1}^{n_l}\sum_{j=1}^{n_d}[D]_{ik}[X]_{kj}=[g]_i \Rightarrow [\tilde{e}]_i\geq 0.
\end{align}
\end{subequations}

In practice, the link and node capacity are typically defined by the number of (ground or air) vehicles, rather than the number of travelers or passengers of the vehicles. Hence, the value of the link and node capacity above often depends on the average number of passengers per ground vehicle and air vehicle. The latter increases, for example, with the capacity of the vehicle and the average level of ridesharing.

\begin{remark}
Several studies in the literature have considered the link capacity constraints \eqref{eqn: link cap} in ground transportation network models \cite{hearn1980bounding}, including the Beckmann model \cite{larsson1999side} and the Nesterov \& de Palma model \cite{nesterov2000stable,nesterov2003stationary}. We refer interested readers to \cite{chudak2007static} for a detailed numerical comparison of the effects of these constraints in different transportation models. 
\end{remark}

\subsection{Traffic equilibria with node and link capacities}

We are now ready to introduce the concept of \emph{static equilibrium matrix}.

\begin{definition}\label{def: Nesterov}
Matrix \(X\in\mathbb{R}^{n_l\times n_v}\) is a static equilibrium matrix defined by the tuple \(\{S, E, D, c, f, g\}\) if it is the optimizer of the linear program in \eqref{opt: flow}.

\vspace{1em}
\noindent\fbox{%
\centering
\parbox{0.96\linewidth}{
{\bf{Linear program for static traffic equilibria}}
\begin{equation}\label{opt: flow}
    \begin{aligned}
    \underset{X}{\mbox{minimize}}\enskip & c^\top X\mathbf{1}_d \\
    \mbox{subject to}\enskip & EX=S, \enskip X\geq 0,\\
    & X\mathbf{1}_d\leq f, \enskip DX\mathbf{1}_d\leq g.
    \end{aligned}
\end{equation}
}%
}

\end{definition}

\begin{remark}
Optimization~\eqref{opt: flow} augments the multicommodity min-cost flow problem \cite[Chp. 4]{bertsekas1998network} with additional node capacity constraints. The main difference between optimization~\eqref{opt: flow} and previous work on the Nesterov \& de Palma model for ground traffic \cite{nesterov2000stable,nesterov2003stationary} is that optimization~\eqref{opt: flow} contains the vertiport capacity constraints in \eqref{eqn: node cap}, which, unlike the link capacity well-studied in the literature, are defined on the nodes of the network rather than the links.
\end{remark}

The linear program in Definition~\ref{def: Nesterov} is our prediction model for the traffic patterns--including flow and travel cost--of a hybrid air-ground transportation network. The following proposition provides two equivalent characterization of static equilibrium matrix based on the optimality condition of linear programs.

\begin{proposition}\label{prop: optimality}
Matrix \(X\in\mathbb{R}^{n_l\times n_v}\) is an static equilibrium matrix associated with the tuple \(\{S, E, D, c, f, g\}\) if and only if there exists \(V\in\mathbb{R}^{n_n\times n_d}\), \(U\in\mathbb{R}^{n_l\times n_d}\), \(p\in\mathbb{R}^{n_l}\), and \(q\in\mathbb{R}^{n_v}\) such that the following two conditions hold simultaneously.
\begin{enumerate}
    \item The following constraints are satisfied:
    \begin{subequations}\label{eqn: primal-dual feasible}
\begin{align}
    &EX=S,\enskip X\mathbf{1}_d\leq f, \enskip DX\mathbf{1}_d\leq g,\label{eqn: current}\\
    &(c+p+D^\top q)\mathbf{1}_d^\top=E^\top V+U,\label{eqn: voltage}\\
    &X\geq 0, \enskip U\geq 0, \enskip p\geq 0, \enskip q\geq 0.\label{eqn: nonneg}
\end{align}
\end{subequations}
    \item One of the following two set of constraints are satisfied: either 
    \begin{equation}\label{eqn: complementary}
    \begin{aligned}
    &\tr(X^\top U)=0, \enskip p^\top X\mathbf{1}_d= f^\top p,\\
    &q^\top DX\mathbf{1}_d= g^\top q,
    \end{aligned}
    \end{equation}
    or 
    \begin{equation}\label{eqn: duality gap}
        c^\top X\mathbf{1}_d+f^\top p+g^\top q=\tr(V^\top S).
    \end{equation}
\end{enumerate}
\end{proposition}
\begin{proof}
See Appendix~\ref{app: prop 1}
\end{proof}

The conditions in \eqref{eqn: complementary} and \eqref{eqn: duality gap} are also known as the \emph{complementary slackness condition} and the \emph{zero-duality-gap condition}. For linear programs, these two conditions are equivalent \cite[Thm. 1.3.3]{ben2001lectures}. Later we will use both conditions to define and simplify the mathematical program with equilibrium constraints for vertiport selection.  

Let \(\tilde{c}=c+p\) and \(\tilde{e}=q\). One can verify that the conditions in \eqref{eqn: current}, \eqref{eqn: nonneg}, and \eqref{eqn: complementary} together imply the constraints in \eqref{eqn: flow conserv}, \eqref{eqn: link cap}, \eqref{eqn: node cap}, \eqref{eqn: link cs}, and \eqref{eqn: node cs}. Hence the equilibria model in Definition~\ref{def: Nesterov} satisfies the second and the third assumptions we introduced at the beginning of this section. 

Furthermore, the conditions in Proposition~\ref{prop: optimality} also implies that only routes with the minimum accumulated travel time are used, a property known as the Wardrop equilibrium principle \cite{wardrop1952road}. To see this implication, we define the set of \emph{route vectors} from node \(i\) to destination node \(s(j)\) as follows:
\begin{equation}
    \mathcal{P}(i, s(j)) = \left\{u\in\{0, 1\}^{n_l}\left|\begin{aligned} &[Eu]_i=1, [Eu]_{s(j)}=-1,\\
    &[Eu]_k=0, \forall k\neq i, s(j).\end{aligned}\right.\right\}.
\end{equation}
Intuitively, each vector \(u\) in set \(\mathcal{P}(i, s(j))\) defines a sequence of links connecting node \(i\) and node \(s(j)\) in a head-to-tail fashion; link \(k\) is on the route defined by \(u\) if and only if \([u]_k=1\). Note that the set \(\mathcal{P}(i, s(j))\) is not necessarily a singleton, since there can be multiple routes--routes composed of ground links, air links, or a combination of both--between each origin-destination pairs.

Based on the above definition, the following corollary shows that any tuple \(\{X, U, V, p, q\}\) satisfying the conditions in Proposition~\ref{prop: optimality} implies that any used routes has the lowest accumulated travel time, where the travel time of link \(k\) is given by \([c+p+Dq]_k\).

\begin{corollary}
\label{cor: Wardrop}
Let \(\{X, U, V, p, q\}\) satisfy the conditions in \eqref{eqn: primal-dual feasible} and \eqref{eqn: complementary}, and \(\overline{c}\coloneqq c+p+Dq\). Let \(i\in\{1, 2, \ldots, n_n\}\) and \(j\in \{1, 2, \ldots, n_d\}\) such that \(i\neq s(j)\) and \([S]_{i, s(j)}>0\). If \(u^\star\in\mathcal{P}(i, s(j))\) and \([X]_{kj}>0\) for all \(k\) such that \([u^\star]_k=1\), then the following condition holds for all \(u\in\mathcal{P}(i, s(j))\):
\begin{equation}
    \overline{c}^\top u^\star \leq \overline{c}^\top u.
\end{equation}
\end{corollary}

Corollary~\ref{cor: Wardrop} shows that the equilibria model in Definition~\ref{def: Nesterov} also satisfies the first assumption we introduced at the beginning of this section: any routes with positive traffic flow has the lowest accumulated time of travel.

Alternatively, one can predict the traffic equilibria using an extension of the Beckmann model, rather than an extension of the Nesterov \& de Palma model \cite{larsson1999side}. However, the Beckmann model result in a set equilibrium conditions with more nonlinear equality constraints than those in Nesterov \& de Palma model \cite{yang1998models,migdalas1995bilevel}: the equilibrium conditions in the Beckmann model are the KKT conditions of a nonlinear convex optimization, which contain nonlinear constraints; in contrast, the equilibrium conditions in the Nesterov \& de Palma model, as we showed in Proposition~\ref{prop: optimality}, only contain linear constraints. On the other hand, studies have shown that the Nesterov \& de Palma model and the Beckmann model give similar prediction results \cite{chudak2007static}. Therefore here we chose the Nesterov \& de Palma model as the basis of our equilibria model.

\section{Vertiport selection via mixed-integer programs}
\label{sec: mixed integer}

We now introduce a mathematical model that selects the location and capacity of vertiports in a hybrid air-ground transportation network as an effort to optimize the resulting traffic equilibria. In particular, we aim to change the optimal solution of linear program \eqref{opt: flow} by choosing the entries in the vertiport-capacity vector \(g\) among discrete values--including zero values, in which case the corresponding vertiport location is discarded.

Throughout we make the following assumptions on linear program~\eqref{opt: flow}.
\begin{assumption}\label{asp: feasible}
Linear program \eqref{opt: flow} is feasible and has a bounded optimal value.
\end{assumption}

Assumption~\ref{asp: feasible} implies that link capacity and vertiport capacity in the hybrid air-ground transportation network are large enough to accommodate the traffic demand, \ie, the flow conservation constraints in \eqref{eqn: flow conserv} and capacity constraints in \eqref{eqn: link cap} and \eqref{eqn: node cap} hold simultaneously. Such an assumption trivially holds in practice, since the ground transportation network alone can accommodate the traffic demand, even without adding any vertiports and air transportation networks.

Based on the above assumption, we will first define the objective function for vertiport selection problem in Section~\ref{subsec: asp & obj}, then define a mathematical program with equilibrium constraints (MPEC) for vertiport selection. We further prove that this MPEC is equivalent to a mixed integer linear program (MILP) in Section~\ref{subsec: MILP}.

\subsection{Vertiport selection via MPEC}
\label{subsec: asp & obj}

We now introduce the mathematical problem for vertiport location and capacity selection. To this end, we first introduce three components of the vertiport selection problem: the design variables, the objective function, and the constraints.  

\subsubsection{The design variables} 
First, we introduce the design variable of the vertiport selection problem. To this end, we start with the following assumption on the vertiport capacity vector \(g\).

\begin{assumption}\label{asp: cap option}
There exists \(G\in\mathbb{R}_+^{n_v\times n_c}\) such that the vertiport capacity vector \(g\) in Definition~\ref{def: Nesterov} satisfies the following constraints:
\begin{equation*}
    [g]_i\in\{0, [G]_{i1}, [G]_{i 2}, \ldots, [G]_{i, n_c}\},
\end{equation*}
where \([G]_{i1}<[G]_{i 2}<\cdots <[G]_{i, n_c}\) for all \(i=1, 2, \ldots, n_v\).
\end{assumption}

Assumption~\ref{asp: cap option} states that the capacity of the \(i\)-th vertiport is selected from an increasing sequence \(\{0, [G]_{i1}, [G]_{i 2},\ldots, [G]_{i, n_c}\}\). For example, if \(n_c=3\), then the capacity of the \(i\)-th vertiport can be zero--in this case, this vertiport is discarded--or a small, medium, or large value, denoted by \([G]_{i1}\), \([G]_{i 2}\), and \([G]_{i3}\), respectively. 

Here we assume the capacity of each vertiport can only be discrete values rather than continuous ones, for the following reasons. First, the capacity of a candidate vertiport--which is the upper bound of the total incoming and outgoing air traffic per unit time--is zero if this candidate is not selected, and strictly positive otherwise. A discrete value of the capacity can capture such as discrete change when a candidate vertiport changes from being not selected to selected. Second, the value of vertiport capacity often can only change discontinuously in practice. For example, increasing the capacity of a vertiport requires increasing the number of touch-down-and-liftoff pads, which can only change discretely rather than continuously. 

Based on Assumption~\ref{asp: cap option}, we define the \emph{selection matrix} as follows. Let \(B\in\mathbb{R}^{n_v\times n_c}\) be a binary matrix such that \([g]_i=[G]_{ij}\) if and only if \([B]_{ij}=1\). Then Assumption~\ref{asp: cap option} holds if and only if
\begin{equation}\label{eqn: B matrix}
     g=(B\odot G)\mathbf{1}_c, \enskip B\mathbf{1}_c\leq \mathbf{1}_v,  \enskip B\in\{0, 1\}^{n_v\times n_c}.
\end{equation}
In other words, each choice of \(B\) that satisfies the constraints in \eqref{eqn: B matrix} corresponds to a value of capacity vector \(g\) that satisfies Assumption~\ref{asp: cap option}. In the following, we will use binary matrix \(B\) as our design variable in the vertiport selection problem.

\subsubsection{The objective function}

Given a set of vertiport with corresponding capacity, we will introduce a quantitative measure for the quality of the traffic equilibria. To this end, given a selected capacity vector \(g\), let \(\{X, U, V, p, q\}\) be a tuple that satisfies the the equilibrium conditions in Proposition~\ref{prop: optimality}. We evaluate the quality of this tuple using the following \emph{network loading function}:
\begin{equation}\label{eqn: loading}
\begin{aligned}
    &\ell(X, p, q)\coloneqq (c+p+Dq)^\top X\mathbf{1}_d\\
    &=\sum_{k=1}^{n_l}\underbrace{[c+p+Dq]_k}_{\overline{c}_k}\underbrace{[X\mathbf{1}_d]_k}_{\overline{x}_k}.
\end{aligned}
\end{equation}
Here the value of \(\overline{c}_k\) is the travel time on link \(k\) at the equilibrium: it is the sum of the free travel time \([c]_k\) and the extra time delay caused by the congestion on the link and nodes, given by \([p+Dq]_k\). The value of \(\overline{x}_k\) is the total amount of travelers entering or exiting link \(k\) per unit time\footnote{At a static equilibrium, the amount of travelers entering and exiting the same link are the same; see \cite{nesterov2003stationary}.}.  

Assumption~\ref{asp: cap option} states that the location and capacity of the vertiports depend on a binary selection matrix \(B\): if \(\sum_{j=1}^{n_c}[B]_{ij}=0\), then candidate vertiport \(i\) is not selected; if \([B]_{ij}=1\), then vertiport \(i\) is selected with capacity \([G]_{ij}\) at the cost of \([K]_{ij}\). In addition, the capacity selection for all the vertiports is subject to a budget constraints defined by parameter \(\gamma\).

\subsubsection{The constraints}
The first set of constraints in our selection problem are given in \eqref{eqn: primal-dual feasible}, \eqref{eqn: complementary} (or \eqref{eqn: duality gap}), and \eqref{eqn: B matrix}.
Together these constraints define the coupling relation among the selection matrix \(B\), the capacity vector \(g\), and the static traffic equilibria that correspond to the tuple \(\{X, U, V, p, q\}\).

In addition, we also consider the following budget and logical constraints on the selection matrix \(B\). First, constructing and maintaining a vertiport comes at a cost--which typically increases with the vertiport capacity. To impose a budget constraints in the vertiport selection problem, we introduce a \emph{cost matrix} \(K\in\mathbb{R}^{n_v\times n_c}\), where its entry \([K]_{ij}\) is the cost of selecting capacity \([G]_{ij}\) for the \(i\)-th vertiport. We let \(\gamma\in\mathbb{R}\) denote the upper bound on the total cost of vertiport selection, then a budget constraint takes the following form:
\begin{equation}\label{eqn: budget}
    \mathbf{1}_v^\top (K\odot B)\mathbf{1}_c\leq \gamma.
\end{equation}

Second, the choice of vertiport location are often subject to additional logical constraints: for example, two locations close to each other cannot be selected simultaneously due to noise management regulations, some locations must be selected as an air traffic hub. To account for these logical constraints, we consider the following linear constraints on the selection matrix \(B\)
\begin{equation}\label{eqn: logical}
    A\vect(B)\leq b,
\end{equation}
where \(\vect:\mathbb{R}^{n_v\times n_c}\to\mathbb{R}^{n_vn_c}\) is a vectorization map such that \([\vect(B)]_{(i-1)n_v+j}=B_{ij}\)  for all \(i=1, 2, \ldots, n_v\) and \(j=1, 2, \ldots, n_c\), \(A\in\mathbb{R}^{n_b\times (n_vn_c)}\) and \(h\in\mathbb{R}^{n_b}\) defines all the logical constraints on matrix \(B\). 

\begin{example}
To illustrate the logical constraints on vertiport location, we consider the case with two candidate vertiport locations, and each vertiport has two candidate capacity value, \ie, \(n_v=n_c=2\). In this case, If we let
\begin{equation}
    A=\begin{bsmallmatrix}
    1 & 1 & 1 & 1\\
    -1 & -1 & -1 & -1
    \end{bsmallmatrix}, \enskip b=\begin{bsmallmatrix}
    1\\
    -1
    \end{bsmallmatrix},
\end{equation}
then the constraint in \eqref{eqn: logical} implies that one and only one of the two candidate vertiport location can be selected, \ie,
\begin{equation}
    [B]_{11}+[B]_{12}+[B]_{21}+[B]_{22}=1.
\end{equation}
\end{example}

\subsubsection{The MPEC for vertiport selection}
\label{subsec: MPEC}

We now introduce a mathematical program that selects the value of capacity vector \(g\). The idea is to optimally choose the value of vector \(g\) such that the resulting equilibrium minimizes a weighted sum of the network loading function in \eqref{eqn: loading} and the selection cost defined in the left hand side of \eqref{eqn: budget}. To this end, we consider the following optimization problem, where \(\omega\in\mathbb{R}_+\) is a weighting parameter. 

\vspace{1em}
\noindent\fbox{%
\centering
\parbox{0.96\linewidth}{
{\bf{Vertiport selection via MPEC}}
\begin{equation}\label{opt: MPEC}
    \begin{aligned}
        \underset{\substack{g, B, p, q,\\
        X, U, V
        }}{\mbox{minimize}}\enskip & (c+p+Dq)^\top X\mathbf{1}_d+\omega\mathbf{1}_v^\top (K\odot B)\mathbf{1}_c \\
        \enskip \mbox{subject to}\,\enskip
        & EX=S,\enskip X\mathbf{1}_d\leq f,\enskip DX\mathbf{1}_d\leq g,\\
        & (c+p+Dq)\mathbf{1}_d^\top=E^\top V+U\\
        & X\geq 0, \enskip U\geq 0, \enskip p\geq 0, \enskip q\geq 0,\\
        & c^\top X\mathbf{1}_d+f^\top p+g^\top q=\tr(V^\top S),\\
        & g=(G\odot B) \mathbf{1}_c, \enskip B\mathbf{1}_c\leq \mathbf{1}_v,\\
    & \mathbf{1}_v^\top (K\odot B)\mathbf{1}_c\leq \gamma, \enskip A\vect(B)\leq b,\\  & B\in\{0, 1\}^{n_v\times n_c}.
    \end{aligned}
\end{equation}
}%
}
\vspace{1em}

Optimization~\eqref{opt: MPEC} is a \emph{mathematical program with equilibrium constraints} (MPEC): it includes the equilibrium conditions in \eqref{eqn: primal-dual feasible} and \eqref{eqn: duality gap} as part of its constraints.  Proposition~\ref{prop: optimality} shows that these constraints--which jointly depend on the primal and dual variables for linear program \eqref{opt: flow}--together ensure that matrix \(X\) is a static equilibrium matrix in the sense of Definition~\ref{def: Nesterov}; similar constraints are common in MPEC, see \cite[Sec. 7.1]{bard2013practical}. According to Proposition~\ref{prop: bilinear}, one can alternatively replace the duality gap constraint in optimization~\eqref{opt: MPEC}--which was first introduced in \eqref{eqn: duality gap}--with the complementarity constraints in \eqref{eqn: complementary}. However, such replacement introduces even more bilinear functions of the unknowns. Hence we choose to write optimization in its current form; a similar MPEC was also used in electrified ground network design \cite{wei2017expansion}.

A global optimal solution of optimization~\eqref{opt: MPEC} is difficult to compute, since its objective function and constraints of optimization~\eqref{opt: MPEC} contains bilinear function of unknowns, such as \(p^\top X\mathbf{1}_d\) and \(g^\top q\).

\subsection{Reformulation of MPEC as an equivalent MILP}
\label{subsec: MILP}

We now show that the MPEC in \eqref{opt: MPEC}, a bilinear mixed integer optimization problem, is equivalent to a mixed integer linear program (MILP). As a result, one can compute a global optimal solution of optimization \eqref{opt: MPEC} using off-the-shelf optimization software, such as GUROBI \cite{gurobi}.  

As our first step, the following proposition shows that,  how to replace the bilinear constraints in optimzation~\ref{opt: MPEC} with a linear one.

\begin{proposition}\label{prop: bilinear}
Let \(G\in\mathbb{R}_{++}^{n_v\times n_c}\). There exists a large enough \(\mu\in\mathbb{R}_{++}\) such that the following two set of conditions are equivalent.
\begin{enumerate}
    \item There exists \(\delta\in\mathbb{R}\), \(q\in\mathbb{R}^{n_v}\), \(B\in\{0, 1\}^{n_v\times n_c}\) and \(g\in\mathbb{R}^{n_v}\) such that
\begin{equation}\label{eqn: bilinear gq}
\begin{aligned}
    & \delta = g^\top q,\enskip g=(G\odot B)\mathbf{1}_c, \enskip B\mathbf{1}_c\leq \mathbf{1}_v, \enskip q\geq 0.
\end{aligned}    
\end{equation}
    \item There exists \(\delta\in\mathbb{R}\), \(q\in\mathbb{R}^{n_v}\), \(B\in\{0, 1\}^{n_v\times n_c}\) and \(Y\in\mathbb{R}^{n_v\times n_c}\), such that
\begin{equation}\label{eqn: linear gq}
\begin{aligned}
    &\delta = \mathbf{1}_v^\top Y\mathbf{1}_c, \enskip 0\leq Y\leq \mu B, \enskip B\mathbf{1}_c\leq \mathbf{1}_v, \\
    &0\leq G\odot (q \mathbf{1}_c^\top)-Y\leq \mu (\mathbf{1}_v\mathbf{1}_c^\top-B),\enskip q\geq 0.
\end{aligned}
\end{equation}
\end{enumerate}
\end{proposition}
\begin{proof}
See Appendix~\ref{app: prop 2}.
\end{proof}

Proposition~\ref{prop: bilinear} allows us to replace the bilinear function \(g^\top q\), appearing in the constraints of optimization~\eqref{opt: MPEC}, with  a linear function of an auxiliary matrix \(Y\).  

Our next step is to show the bilinear objective function of optimization \eqref{opt: MPEC} is also equivalent to a linear one. To this end, by using Proposition~\ref{prop: optimality} again we can show the following:
\begin{equation*}
     p^\top X\mathbf{1}_d=f^\top p,\enskip
    q^\top DX\mathbf{1}_d=g^\top q.
\end{equation*}
Next, thanks to Proposition~\ref{prop: bilinear}, we can further replace the inner product \(q^\top g\) with a linear function of the auxiliary matrix \(Y\). By combining these results together, we can replace the bilinear objective function in  \eqref{opt: MPEC} with a linear one.

Equipped with these results, we can reformulate optimization~\eqref{opt: MPEC} as the following equivalent mixed integer linear program, where \(\mu\) is a large enough positive scalar.

\vspace{1em}
\noindent\fbox{%
\centering
\parbox{0.96\linewidth}{
{\bf{Vertiport selection via MILP}}
\begin{equation}\label{opt: MILP}
    \begin{aligned}
        \underset{\substack{B, p, q, Y,\\
        X, U, V
        }}{\mbox{minimize}}\enskip  & c^\top X\mathbf{1}_d + f^\top p+\mathbf{1}_v^\top Y\mathbf{1}_c+\omega\mathbf{1}_v^\top (K\odot B)\mathbf{1}_c\\
       \mbox{subject to} \enskip
        & EX=S,\enskip X\mathbf{1}_d\leq f,\enskip DX\mathbf{1}_d\leq (G\odot B)\mathbf{1}_c,\\
        & (c+p+D^\top q)\mathbf{1}_d^\top=E^\top V+U\\
        & X\geq 0, \enskip U\geq 0, \enskip p\geq 0, \enskip q\geq 0,\\
        & c^\top X\mathbf{1}_d+f^\top p+\mathbf{1}_v^\top Y\mathbf{1}_c=\tr(V^\top S),\\
        & 0\leq G\odot (q \mathbf{1}_c^\top)-Y\leq \mu (\mathbf{1}_v\mathbf{1}_ {m}^\top-B),\\
        & B\mathbf{1}_c\leq \mathbf{1}_v,\enskip \mathbf{1}_v^\top (K\odot B)\mathbf{1}_c\leq \gamma,\\
        & 0\leq Y\leq \mu B, \enskip A\vect(B)\leq b,\\  & B\in\{0, 1\}^{n_v\times n_c}.
    \end{aligned}
\end{equation}
}%
}
\vspace{1em}

Optimization~\eqref{opt: MILP} is a MILP: its objective function and constraints only depend on linear function of the unknowns, and it contains binary unknown matrix \(B\). One can solve such MILP and obtain a global optimal solution using off-the-shelf optimization software.

For optimization~\eqref{opt: MILP} to be feasible, one needs to choose the value of scalar \(\mu\) to be large enough. In particular, the constraints in Proposition~\ref{prop: bilinear} imply that \(\mu\) needs to be an elementwise upper bound for matrix \(G\otimes (q\mathbf{1}_c^\top)\). Based on Assumption~\ref{asp: cap option}, one can empirically choose \(\mu= \overline{q}\max_i [G]_{i, n_c}\), where \(\overline{q}\in\mathbb{R}_{+}\) is an estimate of the maximum delay among all vertiports at equilibrium.

\section{Numerical experiments}
\label{sec: experiment}

We demonstrate our vertiport selection approach using the Anaheim ground transportation network model developed in \cite{transpnet}, which contains more than 400 nodes and 900 links. Our goal is to numerically demonstrate the effects of adding different vertiports to an existing ground transportation network in terms of traffic loading in the network.

\subsection{The Anaheim transportation network with additional air links}

The Anaheim ground transportation network model consists of a well-defined arterial grid system integrated with an extensive freeway system. See Fig.~\ref{fig: Anaheim net} for an illustration \footnote{The map image we used are generated by Mapbox \url{https://www.mapbox.com}.}. The model includes the data for 1) the incidence matrix, 2) the demand matrix, 3) the free travel time, and 4) the link capacity. Based on these data, we construct the Nesterov \& de Palma model for the ground transportation network, which is known to produce similar results as the Beckmann model \cite{chudak2007static}. 

\begin{figure}
    \centering
    \includegraphics[clip,width=0.9\columnwidth]{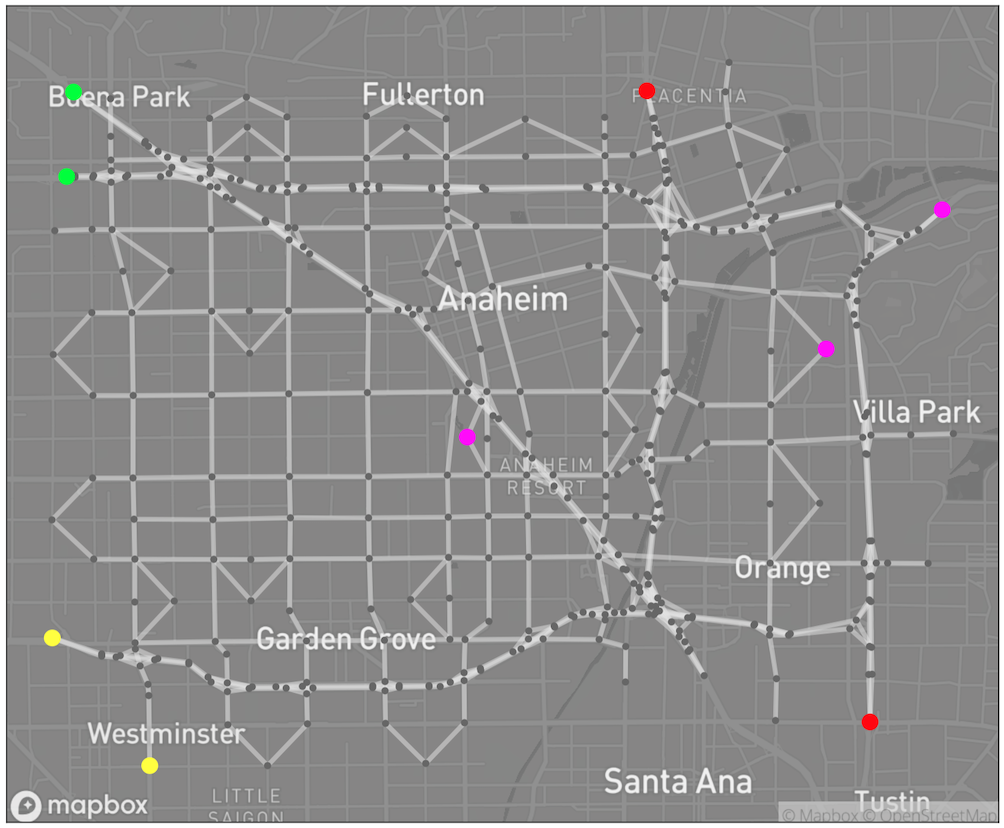}
    \caption{The Anaheim network where the candidate vertiport locations are marked with colored circles.}
    \label{fig: Anaheim net}
\end{figure}

In addition to the Anaheim ground transportation network, we construct an air transportation network as follows. Based on their location and travel demands, we choose 9 different destination nodes in the Anaheim network as candidate location for vertiports; see Fig.~\ref{fig: Anaheim net} for an illustration. The capacity of each vertiport can be either 600 or 1200 takeoffs and landing per hour; choosing these capacity will take \(1\) or \(2\) units of cost. We add an air link to each pair of vertiports if their physical distance is greater than the median of the pairwise distance of all the nodes in the Anaheim network. The free travel time of these air links are set to be proportional to the corresponding distance, and the flow capacity is fixed to be 80 flights per hour for all air links.

We also consider the following budget and logical constraints on the vertiport locations. First, the total selection budget \(\gamma\) is chosen such that \(\gamma\in[5, 11]\). Second, the locations marked in Fig.~\ref{fig: Anaheim net} are subject to the logical constraints listed in Tab.~\ref{tab: logical}. 

\begin{table}[!ht]
    \centering
    \caption{Logical constraints for vertiport locations marked in Fig.~\ref{fig: Anaheim net}}
   \begin{tabular}{c|c}
    \hline 
        Marker color & Constraints on the corresponding locations\\
    \hline     
        red & both are selected \\
        magenta & at least one is selected\\
        yellow & one and only one is selected\\
        green & one and only one is selected\\
    \hline
    \end{tabular}
    \label{tab: logical}
\end{table}

\subsection{Selection based on the Knapsack problem}

As a benchmark approach, we consider selecting vertiport locations using the variation of the Knapsack problem, a classical model in integer programs \cite[Sec. 1.3]{wolsey2020integer}. To this end, 
we define a \emph{value vector} \(w\in\mathbb{R}^{n_v}_+\), where \([w]_k\) denotes the 
value of the unit capacity at the \(k\)-th candidate vertiport. Based on this vector, we compute the selection matrix \(B\) in \eqref{eqn: B matrix} by solving the following mixed integer linear program:
\begin{equation}\label{opt: knapsack}
    \begin{array}{ll}
        \underset{g, B}{\mbox{maximize}} & w^\top g \\
        \mbox{subject to} & g=(B\odot G)\mathbf{1}_m, \enskip B\mathbf{1}_m\leq \mathbf{1}_v,\\
        & \mathbf{1}_v^\top (K\odot B)\mathbf{1}_m\leq \gamma, \enskip A\vect(B)\leq b,\\
        & \enskip B\in\{0, 1\}^{n_v\times n_m}.
    \end{array}
\end{equation}
Notice that optimization~\eqref{opt: knapsack} contains the the discrete capacity constraints in \eqref{eqn: B matrix}, the budget constraints in \eqref{eqn: budget}, the logical constraints in \eqref{eqn: logical}.

The difficulty in using optimization~\eqref{opt: knapsack} for vertiport selection is the estimation of the value vector.  Here we consider a heuristics estimate by choosing the elements in vector \(w\) to be the total traffic demand at the candidate vertiport; the idea behind this heuristics is that the value of the unit capacity at a candidate vertiport should increase with the travel demand: the higher the demand, the more beneficial to provide air travel as an alternative. In particular, we choose the candidate vertiport nodes \(\mathcal{V}=\{v(1), v(2), \ldots, v(n_v)\}\) from the set of destination nodes \(\{s(1), s(2), \ldots, s(n_d)\}\) such that there exists \(1\leq d \leq n_d\) with \(v(k)=s(d)\) for all \(k=1, 2, \ldots, n_v\). Furthermore, we let \([w]_k=S(s(d), d)\).

\subsection{Numerical comparison}

With the above choices of parameters, we solve optimization~\eqref{opt: MILP}.  To demonstrate our results, we define the following notion of \emph{link loading} for each link \(k=1, 2, \ldots, n_l\):
\begin{equation}
    \ell_k(X, p, q)=[c+p+Dq]_k[X\mathbf{1}_d]_k. 
\end{equation}
Intuitively, \(\ell_k\) denotes the number of vehicles traveling on link \(k\) at the equilibrium--which is also the summand in the total link loading defined in \eqref{eqn: loading}.

Fig.~\ref{fig: Anaheim loading} shows the link loading in the ground and air networks when we let choose the budget to be \(\gamma=8\). In this case, a total of six vertiports are selected, and only two of them has the larger capacity value 1200: the one near Westminster and the one near Villa Park; the latter fact is consistent with the air link loading distribution in Fig.~\ref{fig: Anaheim loading}: the vertiports near Westminster and Villa Park are connecting some of the flight legs with the highest loading, hence they necessarily need larger capacity.

\begin{figure}[!htp]
\centering
\subfloat[Air traffic network loading.]{%
  \includegraphics[clip,width=0.9\columnwidth]{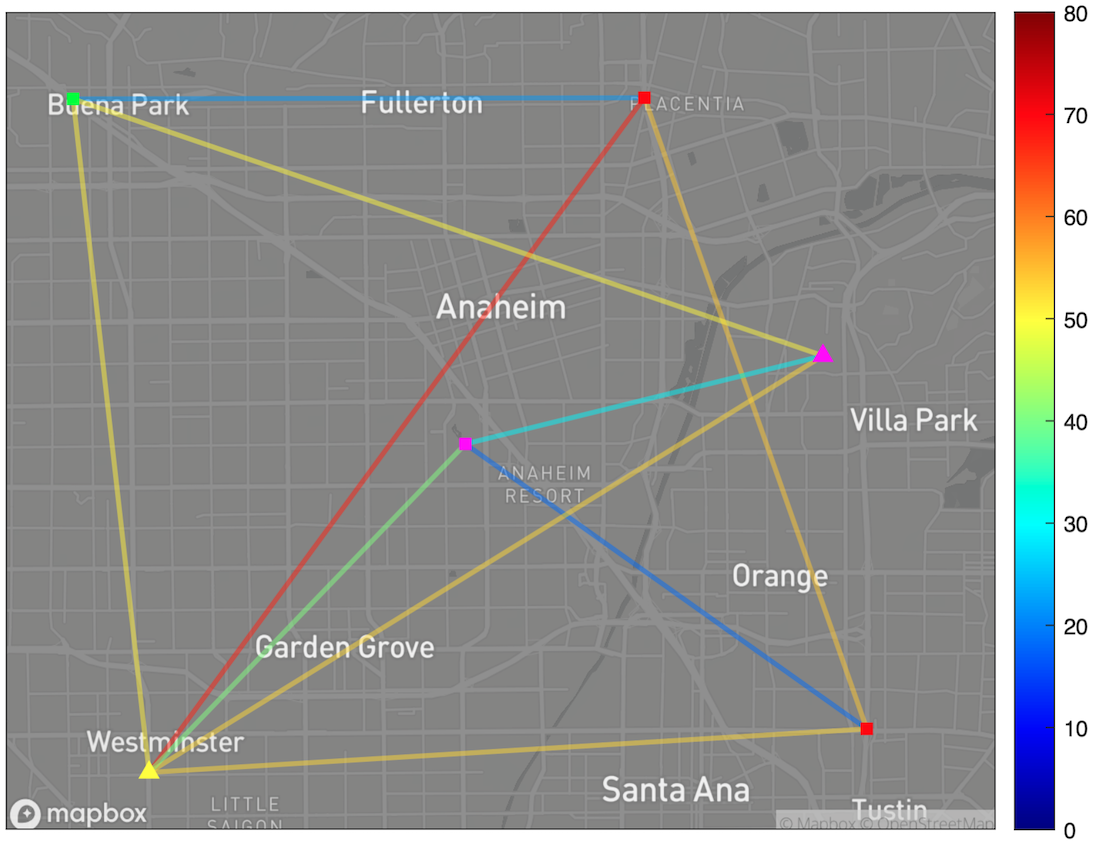}%
}

\subfloat[Ground traffic network loading.]{%
  \includegraphics[clip,width=0.9\columnwidth]{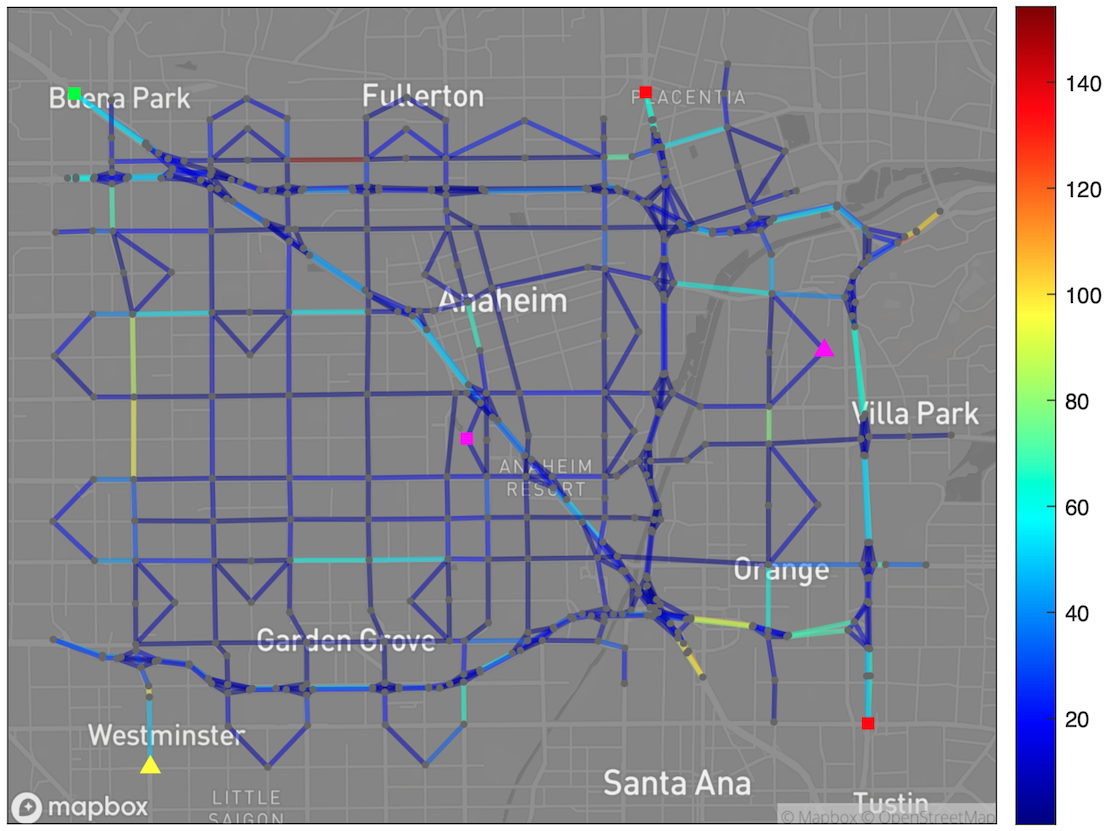}%
}

\caption{The optimal air and ground traffic network loading when vertiport selection budget \(\gamma=8\). The shape of the marker indicates the capacity of the corresponding vertiport: square marker means capacity value 600, costs one unit in the budget; triangle markers means capacity value 1200, which costs 2 units in the budget.}
\label{fig: Anaheim loading}
\end{figure}

We also show how does the budget value \(\gamma\) in vertiport selection affect the link loading in the ground traffic network. Intuitively, adding vertiports will reduce the ground link loading by providing alternative means of transportation. Furthermore, as the budget increases, the selected vertiports can support an air transportation network with larger volume of air traffic, and consequently, the ground link loading will decrease more. These intuitions are confirmed by Fig.~\ref{fig: curve} and Fig.~\ref{fig: bar}, which shows the sum of the link loading reduction in the ground network increases with the budget value, and so does the number of ground links with decreased loading. 

We also compare the performance of the results based on the in \eqref{opt: MILP} and the results based on the Knapsack problem in \eqref{opt: knapsack}; both of which contain 18 binary integer variables in this problem. Fig.~\ref{fig: curve} and Fig.~\ref{fig: bar} show that the MPEC approach is better than the Knapsack problem approach in terms of the total link loading reduction in the ground network as well as the number of ground links with decreased loading. These results confirm the advantage of the MPEC approach. We note that it may be possible for the Knapsack problem to produce results similar to those of the MPEC approach, via a better estimate of the value vector--rather than directly using the total travel demand--in optimization~\eqref{opt: knapsack}. However, to our best knowledge, there is no systematic method to compute these estimates. Hence MPEC is more useful for vertiport selection.

\begin{figure}[!htp]
    \centering
    \includegraphics[clip,width=\columnwidth]{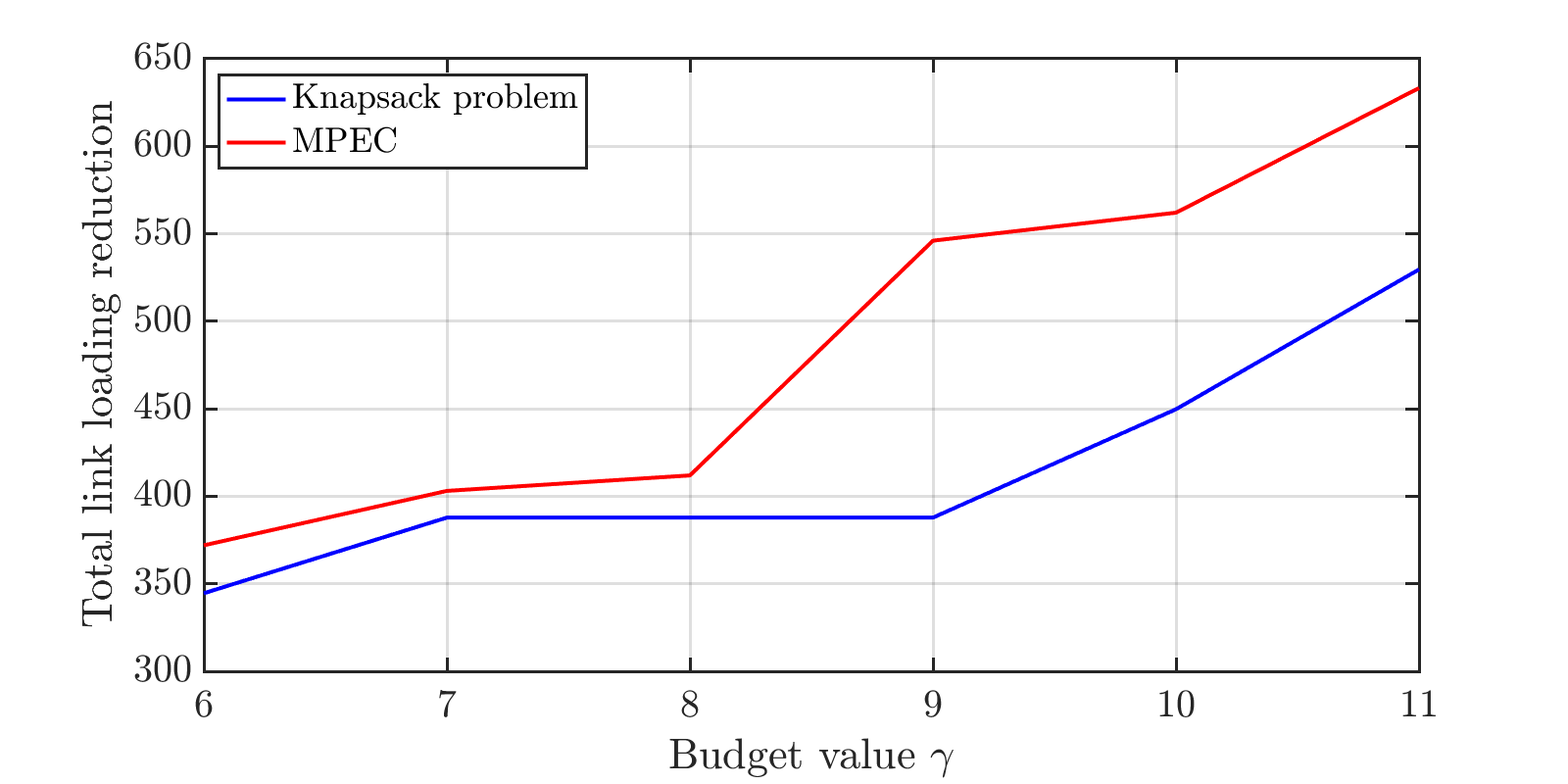}%
    \caption{The total link loading reduction in the ground network due to vertiport addition for different budget value \(\gamma\): a comparison between the MPEC approach and the Knapsack problem approach.}
    \label{fig: curve}
\end{figure}

\begin{figure}[!htp]
    \centering
    \includegraphics[clip,width=\columnwidth]{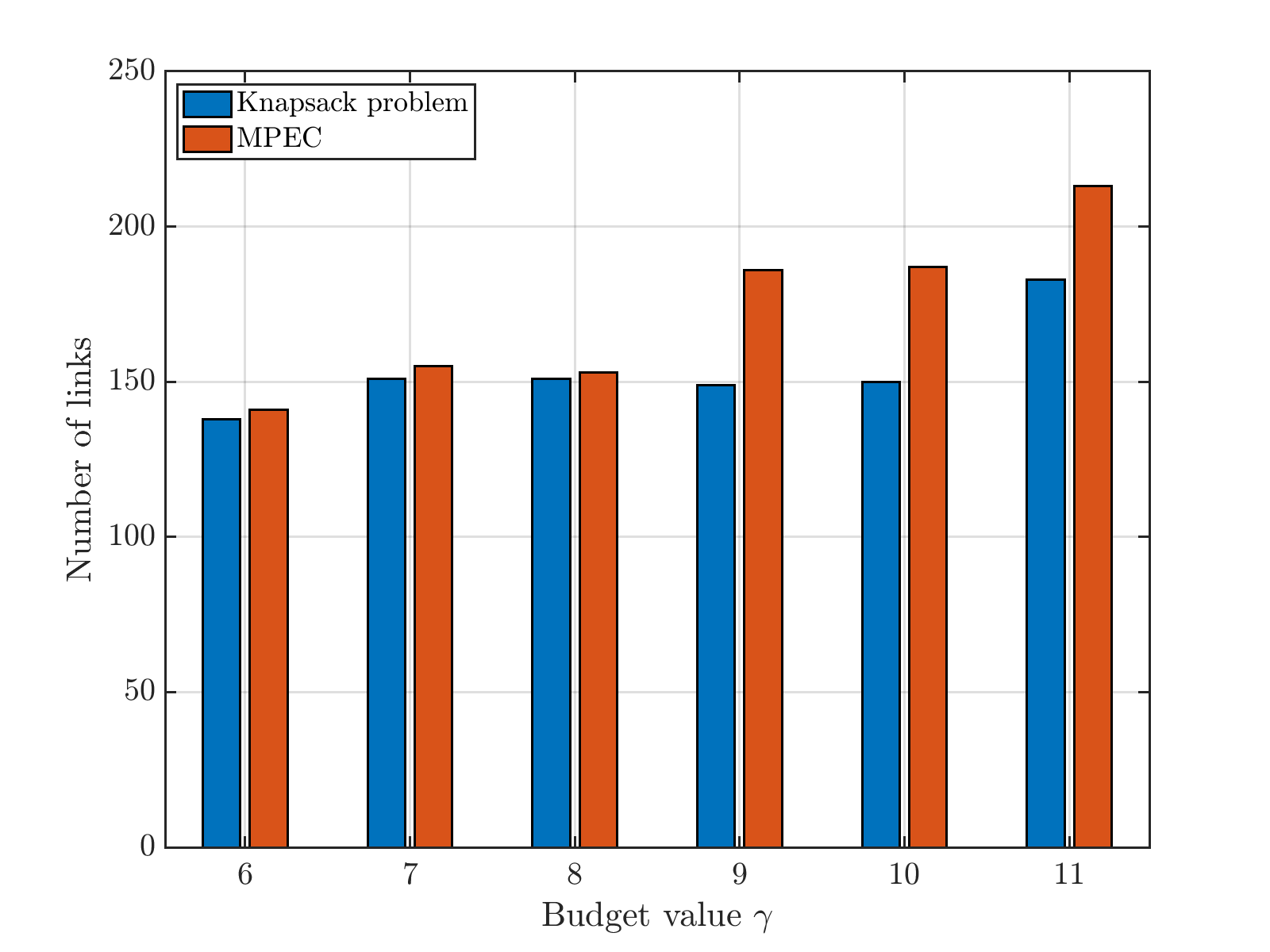}%
    \caption{The number of ground links whose link loading decrease due to vertiport addition with different budget value \(\gamma\): a comparison between the MPEC approach and the Knapsack problem approach.}
    \label{fig: bar}
\end{figure}

\section{Conclusion}
\label{sec: conclusion}

We introduce a mathematical model to select the optimal vertiport location and capacity for minimizing the traffic congestion in a hybrid air-ground transportation network. Our model is equivalent to a mixed-integer linear program, and we demonstrate this model using the Anaheim transportation network. 

Our work also opens some new research questions. For example, although the identification of the parameters for ground transportation networks--such as free travel time and link capacity--are well studied in the literature, similar results are still missing for the air transportation networks. In order to use the mathematical models we developed, it is important to identify these parameters using realistic air traffic data. Another example is to consider the impacts of different weather conditions in the vertiport selection problem. Since weather conditions are more likely to affect the operation of aircraft than automobiles, it is critical to ensure the air transportation network is robust against temporary capacity decrease caused by extreme weather conditions. We aim to answer these open questions in our future work.

\section{Appendix}

\subsection{Proof of Proposition~\ref{prop: optimality}}
\label{app: prop 1}

We start by deriving the dual of linear program \eqref{opt: flow}. Let the \emph{Lagrangian} be defined as
\begin{equation}\label{eqn: Lagrangian}
\begin{aligned}
    &L(X, U, V, p, q)=c^\top X\mathbf{1}_{d}-\tr(V^\top EX)+\tr(V^\top S)\\
    &-\tr(U^\top X)+p^\top (X\mathbf{1}_{d}-f)+q^\top (DX\mathbf{1}_{d}-g).
\end{aligned}
\end{equation}
The dual of linear program is given by 
\begin{equation}\label{opt: dual flow}
\begin{aligned}
  \underset{U, V, p, q}{\mbox{maximize}}\enskip & \psi(U, V, p, q)\\
  \mbox{subject to}\quad & U\geq 0, \enskip p\geq 0, \enskip q\geq 0.
\end{aligned}
\end{equation}
where \(\psi(U, V, p, q)=\min_X L(X, U, V, p, q)\). Since matrix trace is invariant under cyclic permutation, we have
\begin{equation*}
    \begin{aligned}
      & c^\top X\mathbf{1}_{d}=\tr(\mathbf{1}_{d}c^\top X),\enskip p^\top X\mathbf{1}_{d}=\tr(\mathbf{1}_{d}p^\top X),\\
      & q^\top DX\mathbf{1}_{d}=\tr(\mathbf{1}_{d}q^\top D X).
    \end{aligned}
\end{equation*}
Substitute the above equalities into \eqref{eqn: Lagrangian}, we can show the following
\begin{equation*}
\begin{aligned}
    &\frac{\partial}{\partial X} L(X, U, V, p, q)\\
    &=\frac{\partial}{\partial X}\tr( (\mathbf{1}_{d}(c^\top +p^\top +q^\top D)-V^\top E-U^\top) X)\\
    &=(c+p+Dq)\mathbf{1}_{d}-E^\top V-U
\end{aligned}
\end{equation*}
Since \(L(X, U, V, p, q)\) is a linear function of \(X\), we have \(\psi(U, V, p, q)=L(X, U, V, p, q)\) if and only if \(\frac{\partial}{\partial X} L(X, U, V, p, q)=0\). Therefore we can rewrite optimization \eqref{opt: dual flow} equivalently as follows
\begin{equation}\label{opt: dual flow 2}
\begin{aligned}
  \underset{U, V, p, q}{\mbox{maximize}}\enskip & \tr(V^\top S)-f^\top p-g^\top q\\
  \mbox{subject to}\quad & (c+p+Dq)\mathbf{1}_{d}^\top=E^\top V+U\\
  &U\geq 0, \enskip p\geq 0, \enskip q\geq 0.
\end{aligned}
\end{equation}
Using \cite[Thm. 1.3.3]{ben2001lectures}, we conclude that \(X\) and \(U, V, p, q\) are optimal for linear program \eqref{opt: flow} and \eqref{opt: dual flow 2}, respectively, if and only if the primal and dual feasibility condition in \eqref{eqn: primal-dual feasible} and the complementary slackness condition \eqref{eqn: complementary} are satisfied. Furthermore, the complementary slackness conditions in \eqref{eqn: complementary} are equivalent to the zero duality gap condition in \eqref{eqn: duality gap}.

\subsection{Proof of Corollary~\ref{cor: Wardrop} }

Since \(u^\star, u\in\mathcal{P}(i, s(j))\), by pre-multiplying equation~\eqref{eqn: primal-dual feasible} with \(u^\star\) and \(u\) and we can show the following:
\begin{subequations}\label{eqn: two route}
\begin{align}
    &\textstyle (u^\star)^\top \overline{c}=V_{ij}-V_{s(j), j}+\sum_{k=1}^{n_l}[u^\star]_j[U]_{kj},\\
   &\textstyle u^\top \overline{c}=V_{ij}-V_{s(j), j}+\sum_{k=1}^{n_l}[u]_j[U]_{kj}.
\end{align}
\end{subequations}
In addition, the constraints in \eqref{eqn: nonneg} and \eqref{eqn: complementary} together implies that \([U]_{kj}=0\) for all \(k\) such that \([X]_{kj}>0\). Combining this fact with the assumption that \([X]_{kj}>0\) for all \(k\) such that \([u^\star]_k=1\), we conclude that \([U]_{kj}=0\) for all \(k\) such that \([u^\star]_k=1\). Hence
\begin{equation}\label{eqn: slackness}
   \textstyle(u^\star)^\top \overline{c}=V_{ij}-V_{s(j), j}+\sum_{k=1}^{n_l}[u^\star]_j[U]_{kj}=V_{ij}-V_{s(j), j}.
\end{equation}
By combining \eqref{eqn: two route} and \eqref{eqn: slackness}, we obtain the following
\begin{equation*}
       \textstyle(u^\star)^\top \overline{c}=V_{ij}-V_{s(j), j}=u^\top \overline{c}-\sum_{k=1}^{n_l}[u]_j[U]_{kj}\leq u^\top \overline{c},
\end{equation*}
where the last step is because \(u\) and \(U\) are both elementwise nonnegative.

\subsection{Proof of Proposition~\ref{prop: bilinear}}
\label{app: prop 2}

First, suppose \(\delta, q, B\) and \(g\) satisfy the constraints in \eqref{eqn: bilinear gq}. Let \([Y]_{ij}=[g]_i[q_t]_i[B]_{ij}\) for all \(i=1, 2, \ldots, n_v\) and \(j=1, 2, \ldots, n_m\), and \(\mu= \max_{i, j}\,[q]_i [G]_{ij}\). Then one can verify that \(\delta, q, B\) and \(Y\) satisfy the constraints in \eqref{eqn: linear gq}.

Second, suppose \(\delta, q, B\) and \(Y\) satisfy the constraints in \eqref{eqn: linear gq} for some sufficiently large \(\mu\in\mathbb{R}_{++}\). The constraints \(B\in\{0, 1\}^{n_v\times n_m}\) and \(B\mathbf{1}_{n_m}\leq \mathbf{1}_{n_v}\) implies that each row of matrix \(B\) can have at most one entry equals one. Hence we can obtain an unique vector \(g\) by defining its \(i\)-th entry as follows:
\begin{equation}\label{eqn: g entry}
    [g]_i=\begin{cases}
    [G]_{ij}, & \text{if \([B]_{ij}=1\)},\\
    0, & \text{if \([B]_{ij}=0\) for all \(j=1, 2, \ldots, n_m\).}
    \end{cases}
\end{equation}
Next, since \(\mu\in\mathbb{R}_{++}\) is sufficiently large, an upper bound of \(\mu\) can be treated as redundant. As a result, if \([B]_{ij}=0\), then the constraints in \eqref{eqn: linear gq} implies that \([Y]_{ij}=0\) and \([G]_{ij}[q]_i\geq 0\). Since \(q\geq 0\) and \(G\geq 0\), the latter constraint is redundant. Furthermore, if \([B]_{ij}=1\), then the constraints in \eqref{eqn: linear gq} implies that
\[0\leq [Y]_{ij}, \enskip [G]_{ij}[q]_i=[Y]_{ij}.\]
By combining the above two cases with the definition in \eqref{eqn: g entry}, we conclude that \(\textstyle \sum_{i=1}^{n_v}\sum_{j=1}^{n_m}[Y]_{ij}=\sum_{i=1}^{n_v}[g]_i[q]_i\)
for all \(i=1, 2, \ldots, n_v\) and \(j=1, 2, \ldots, n_m\). Therefore, \(\delta, q, B\) and \(g\) satisfy the constraints in \eqref{eqn: bilinear gq}.



\section*{ACKNOWLEDGMENT} The authors would like to thank Rishabh Thakkar and Jorge Martinez Zapico for their help in obtaining traffic data, and Aditya Deole, Shahriar Talebi, and Kuang-Ying Ting for helpful discussions. 

\bibliographystyle{IEEEtran}
\bibliography{IEEEabrv,reference}

\end{document}